\documentclass[11pt]{article}
\usepackage{mathbbold}

\usepackage{amsmath}
\usepackage{amsfonts}
\usepackage{amssymb}
\usepackage{graphicx}
\usepackage{multirow}
\usepackage{tikz}
\usepackage{enumerate}
\usepackage[all]{xy}

\def\lg{{\rm \langle}} \def\rg{{\rm \rangle}}
  \def\sbs{{\rm \subseteq}}

\def\wi{\widetilde}

\def\Aut{ {\rm Aut}}
\def\Inn{ {\rm Inn}}

\def\GC{{\rm GC}}
\def\Cay{{\rm Cay}}

\allowdisplaybreaks

\begin{document}

\newtheorem{theorem}{Theorem}[section]
\newtheorem{corollary}[theorem]{Corollary}
\newtheorem{definition}[theorem]{Definition}
\newtheorem{conjecture}[theorem]{Conjecture}
\newtheorem{question}[theorem]{Question}
\newtheorem{lemma}[theorem]{Lemma}
\newtheorem{proposition}[theorem]{Proposition}
\newtheorem{quest}[theorem]{Question}
\newtheorem{example}[theorem]{Example}
\newenvironment{proof}{\noindent {\bf
Proof.}}{\rule{2mm}{2mm}\par\medskip}
\newcommand{\remark}{\medskip\par\noindent {\bf Remark.~~}}
\newcommand{\pp}{{\it p.}}
\newcommand{\de}{\em}

\title{  {The classification of local $m$-GCI-group on finite nonabelian simple groups}
\thanks{  E-mail addresses:  xiaominzhu@sjtu.edu.cn (X. Zhu), xcubicy@163.com (X.
Yang, corresponding author).}}

\author{
Xiao-Min Zhu$^{a}$, Xu Yang$^{b}$  \\
{\small $^a$School of Mathematical Sciences, Shanghai Jiao Tong University,}\\
{\small 800 Dongchuan Rd, Shanghai, China.}\\
{\small $^b$School of Statistics and Mathematics, Shanghai Lixin University of Accounting and Finance,} \\
{\small 995 Shangchuan Rd, Shanghai, China. }\\
}

\maketitle

\vspace{-0.5cm}

\begin{abstract}
Li and Praeger classified  finite nonabelian simple groups, it has only one or two fusion classes of any certain value. As a by-product, they classified $m$-CI-groups, which is critical in the research of Cayley graphs.
In the paper, we will consider generalized Cayley graphs. This concept is proposed by Maru\v{s}i\v{c} et al. In the paper, (local) $m$-GCI-group is defined, and we get many properties and characterizations based on the generalized Cayley isomorphism, which are the key measures for the classification of (local) $m$-GCI-group. And above all, we will give a classification of local 2-GCI-groups and 2-GCI-groups for finite nonabelian simple groups.
\end{abstract}

{{\bf Key words:} fusion class; generalized Cayley graph; finite nonabelian simple groups; local 2-GCI-groups}

{{\bf MR(2010) Subject Classification:} 05C25}

\section{Introduction}

\ \ \quad
The graphs considered to be finite, simple and undirected, and groups
are nonabelian simple groups without special statement.
For finite group $G$, the conjugation is defined by $a^b=bab^{-1}$ for any $a,b\in G$. The full automorphism group is denoted by $\Aut(G)$ and $1$ is denoted to be the identity of $G$.
Elements $a$ and $b$ (rep. $b^{-1}$) are {\it fused} (rep. {\it inverse-fused}) if there is some automorphism mapping $a$ to $b$ (rep. $b^{-1}$).
Let $\mathcal{F}_a:=\{a^\gamma|\gamma\in \Aut(G)\}$. Then we call $\mathcal{F}_a$ the {\it fusion class} of $a$ of order $o(a)$.
Let $F_G(n)$ (rep. $C_G(n)$) be the number of fusion classes (rep. conjugacy classes) of order $n$. Similarly, we have $\mathcal{F}_H=\{H^\gamma|\gamma\in\Aut(G)\}$ for $H\leq G$.

Suppose that $P_e(n,i)\,({\rm rep.}\, P_s(n,i))$ denotes the property that all elements (rep. subgroups) of $G$ of order $n$ are divided into at most $i$ conjugacy classes in $G$ or fusion classes in $\Aut(G)$.
In history, the research of groups having $P_e(n,i)$ or $P_s(n,i)$ with small $i$
has been attracted much attention. Let $p$ be an odd prime. Shult \cite{EE} showed
that if $p$-group $G$ has $P_s(p,1)$ in $\Aut(G)$, then $G$ is homocyclic .
When $p=2$, the situation is complicated, Gross \cite{FG}, Higman \cite{GH}, and Shaw \cite{DS} gave a description of
$2$-groups $G$ that $G$ has $P_e(2,1)$ in $\Aut(G)$. For a $p$-solvable groups $G$, Gasch\"{u}tz and Yen \cite{WG-TY}
proved $G$ has $p$-length one if $G$ has $P_e(p,1)$ in $\Aut(G)$.
In \cite{WF-GMS}, Feit and Seitz solved an open problem \cite{note} they obtained
that $G$ is isomorphic to ${\rm S}_2$ or ${\rm S}_3$ if $G$ has $P_e(n,1)$ in $G$ for any $n\big{|}|G|$,
where ${\rm S_n}$ is a symmetric group of degree $n$.
In 1992, Zhang \cite{JZ} gave a crucial characterization of $G$ with $P_e(n,1)$
for any $n\big{|}|G|$ in $\Aut(G)$. In 1994, Li \cite{Lich} obtained a classification of $G$ if $G$ has
$P_e(n,2)$ for any $n\big{|}|G|$ in $G$. In
1996, Li and Praeger \cite{Lich-CP-1} researched $G$ which are simple and finite if $G$ has
$P_e(n,2)$ for any $n\big{|}|G|$ in $\Aut(G)$. In 1997,
they investigated the same problem for all finite groups and obtained classification \cite{Lich-CP-2}.


Let $X=\Cay(G,S)$, $S$ is a {\it Cayley subset}. Then $X$ has vertex-set $G$ and edge-set
$\{\{x,y\}|xy^{-1}\in S\}$. If $S$ is self-inverse, then $S=S^{-1}:=\{x^{-1}|x\in S\}$, it means $X$ is undirected. As is well-known, isomorphism problem is one of the fundamental problems for algebraic and combinatorial structure, the problem on Cayley graphs has attracted much attention in the history \cite{fyq,Lich-CP-1,Lich-CP-2,Li-CP1999,toida1977}.
Let $S_i$ be any self-inverse subset such that $|S_i|\leq m,\, (i=1,\,2)$. Let $X_i=\Cay(G,S_i)$. Then $G$ is a {\it $m$-CI-group} if $X_1\cong X_2$ and $S_1^\gamma=S_2$ are valid for some $\gamma\in\Aut(G)$. In \cite{toida1977}, Toida showed that
cyclic groups are 3-CI-groups. When $m\geq 4$, the authors \cite{Li-CP1999} investigated finite groups containing all $m$-CI-groups.
In \cite{Lich-CP-1}, they noticed the relationship between groups with a few fusion
classes and the $m$-CI-groups ($m\leq3$). As a result, they classified $m$-CI-groups ($m\leq3$) on
finite nonabelian simple groups.
In detail, they introduced:
$$\Omega(G,n)=\{\{x,x^{-1}\}|x\in G, o(x)=n\},$$
and solved questions below:
\begin{question}\label{ques1}
Let $G$ be simple and finite. Can we classify $G$ such that $G$ has
$P_e(n,2)$ in $\Aut(G)$ for any $n\big{|}|G|$ ?
\end{question}
\begin{question}\label{ques1-1}
Let $G$ be simple and finite. Can we classify $2$-CI-groups about $G$?
\end{question}

The above questions are the principal motivation of the paper.
In \cite{MSS1992}, generalized Cayley graph is proposed by Maru\v{s}i\v{c} et al.
\begin{definition}\cite{MSS1992}\label{D-G}
Let $G$ be finite and $\alpha\in\Aut(G)$.  Suppose that $S\subset G$. If $G,S,\alpha$ satisfy
\begin{enumerate}[(1).]
  \item $\alpha^2=1$;
  \item if $x\in G$, then $\alpha(x)x^{-1}\notin S$;
  \item if $\alpha(x^{-1})y\in S$, then $\alpha(y^{-1})x\in S$ for $x,y\in G$,
\end{enumerate}
then there exists generalized Cayley graph $X$ with vertex-set $G$ and edge-set $\{\{x,y\}|\alpha(x^{-1})y\in S\}$, denoted by $\GC(G,S,\alpha)$.
\end{definition}

By Definition \ref{D-G}, (2) certifies that the graph has no loops, (3) certifies the graph is
undirected and $\alpha(S)=S^{-1}$, where $S$ is called a
{\it generalized Cayley subset} with respect to $\alpha$.
More details about generalized Cayley graphs can refer to \cite{HKM2015,HKPT2017,H2019,MSS1992,YLCF,YLF}. Note that $\alpha^2=1$
by (1). If $\alpha=1$, then $\GC(G,S,\alpha)$ is $\Cay(G,S)$.

In \cite{YLF}, the authors proposed  GCI-groups and restricted GCI-groups (GCI stands for generalized Cayley isomorphism). Considering the contribution of $m$-CI-groups to CI-groups, here we propose the following concept:
\begin{definition}\label{mgci}
Suppose $\alpha_i\in\Aut(G)$, where $\alpha_i^2=1\,(i=1,2)$ for finite group $G$. Let $X_i=\GC(G,S_i,\alpha_i)$ be any two isomorphic graphs, where where $|S_i|\leq m$.
\begin{enumerate}[(1).]
  \item If there always exist $x\in G$ and $\gamma\in\Aut(G)$ such that $\alpha_1^\gamma(x)S_1^\gamma x^{-1}=S_2$ and $\alpha_2=\alpha_1^\gamma$,
  then we call $G$ a $m$-GCI-group;
  \item If $\alpha_i\,(i=1,2)$ are involutory automorphisms and the condition in (1) is satisfied, then we call
$G$ a local $m$-GCI-group.
\end{enumerate}
\end{definition}



According to Definition \ref{mgci}, if $G$ is $m$-GCI-group, then we can find group $G$ is local $m$-GCI-group.
However, the inverse is not necessarily true.
In order to classify $m$-GCI-groups and local $m$-GCI-groups, we propose two questions below:
\begin{question}\label{ques2}
Let $G$ be simple and finite. Can we classify $G$ such that some setwise stabilizer of some subset of $G$ in $\Aut(G)$ has $P_e(n,2)$ for any $n\big{|}|G|$ ?
\end{question}

\begin{question}\label{ques2-1}
Let $G$ be simple and finite. Can we classify local $2$-GCI-groups and $2$-GCI-groups on $G$?
\end{question}

Note that if the setwise stabilizer is trivial, then Question \ref{ques2} is actually a generalization of Question \ref{ques1}. Furthermore, Question \ref{ques2} is the key measures to Question \ref{ques2-1}, that is the classification of local $2$-GCI-groups and $2$-GCI-groups. The following theorem is one of the main result.

\begin{theorem}\label{main-re}
Let $G$ be simple and finite.
\begin{enumerate}
  \item $G$ is a local 2-GCI-group if and only if $G$ is one of $A_5$, $L_2(8)$, $Sz(8)$, $M_{11}$ or $M_{23}$;
  \item $G$ can not be a $m$-GCI-group ($m\geq 1$).
\end{enumerate}
\end{theorem}

We will construct the paper below. In the next section, some notaions and `fusion class' with respect to generalized Cayley graphs will be given. In Section \ref{answer1}, we will give a proof of Question \ref{ques2}. In Section \ref{example}, some examples of  groups which are not local $2$-GCI-groups are obtained. In Section \ref{answer2}, Theorem \ref{main-re} will be proved, which is the complete classification of local $2$-GCI-groups and $2$-GCI-groups on finite nonabelian simple groups.

\section{Preliminaries}\label{pre}
Let $\alpha\in\Aut(G)$ be an involution. Following notations from \cite{HKM2015,MSS1992,PS}, we write
$$G_{\pm\alpha}:=\{x\in G|x^\alpha=x^{\pm1}\}.$$
Moreover, $\Aut(G)_x=\{\gamma\in\Aut(G)|x^\gamma=x\}$ where $x\in G$.
Let $\omega_\alpha(G)=\{\alpha(x)x^{-1}|x\in G\}$. If $\alpha$ is inner, like $\alpha=\sigma_x: g\mapsto xgx$ for some $x\in G$, then we denote $\omega_\alpha(G)$ by $\omega_x(G)$, and $\omega_x(G)=\{\sigma_x(g)g^{-1}|g\in G\}$.
In fact, we can find that $\omega_\alpha(G)$ and $G_{-\alpha}$ are self-inverse, and $\omega_\alpha(G)\sbs G_{-\alpha}$ by \cite{YLF}.
Let $\omega^*_x(G)=\{\{xg, g^{-1}x^{-1}\}|g\in \omega_\alpha(G)\}$.
Then $\omega^*_x(G)$ is the conjugacy class of $x$. Then we have a property as follows:

\begin{lemma}\label{+1}
Let $C_G(2)=1$ for a finite group $G$. Then $|G_{-\alpha}|=|\omega_\alpha(G)|+1$ for any involutory automorphism $\alpha$ induced by an involution.
\end{lemma}
\begin{proof}
Let $\alpha=\sigma_x$. Then $G_{-\alpha}=\{g\in G|(xg)^2=1\}$. It follows that $xg=1$ or $xg$ is an involution. Therefore, $|G_{-\alpha}|=1+|C_{xg}|=1+|C_{x}|$, where $C_{x}$ is the conjugacy class of $xg$. Note that $|\omega_\alpha(G)|=\frac{|G|}{|G_\alpha|}$, if $\alpha=\sigma_x$, then
$|\omega_\alpha(G)|=\frac{|G|}{|C_G(x)|}=|C_x|$. This completes the proof.
\end{proof}

By (2) in Definition \ref{D-G}, we define the following multi-subset,
$$\Pi_\alpha(G)= \{\{x,\alpha(x^{-1})\}|x\in G, x\notin \omega_\alpha(G)\}.$$
If $S\subset G$ is a generalized Cayley subset, then it is actually the union of some subsets in $\Pi_\alpha(G)$. Let $x\in G$. Then equation $h=x\cdot g$ always has a solution in $G$. Intend to highlight $x$, we can denote $\Omega(G,n)$ for $\Omega_x(G,n)$ and have:
$$\Omega_x(G,n)=\{\{xg,g^{-1}x^{-1}\}|g\in G, o(xg)=n\}.$$
In \cite{Lich-CP-1}, the authors got the following theorems:


When $\alpha=\sigma_x$ for some $x\in G$, we also denote $\Omega_x(G,n)$ by $\Omega_\alpha(G,n)$ and denote $\Omega_x^*(G,n)$ by $\Omega_\alpha^*(G,n)$. In detail, $\Omega_x^*(G,n)$ is defined as follows:
\begin{equation}\label{equation-doubi}
\Omega_x^*(G,n)=\left\{\begin{array}{ll}
                          \Omega_x(G,n), & n\neq 2, \\
                          \Omega_x(G,2)-\omega^*_x(G), & n=2.
                        \end{array}\right.
\end{equation}

Let $\GC(G,S_1,\alpha_1)\cong\GC(G,S_2,\alpha_2)$. If it is a generalized Cayley isomorphism, then we have
$\alpha_1^\gamma(x)S_1^\gamma x^{-1}=S_2$ and $\alpha_1^\gamma=\alpha_2$ for some $\gamma\in\Aut(G)$ and $x\in G$. Hence, there must exist
$a\in S_1$ and $b\in S_2$ satisfying $\alpha_1(x)a^\gamma x^{-1}=b$. Thus we get the following definition similar to the concept of fused and inverse-fused.

\begin{definition}\label{qfiqf}
Let $\alpha\in\Aut(G)$ be an involution. For $a,b\in G$,
if there exists $x\in G$ such that  $b=\alpha(x)a^\gamma x^{-1}$ (rep. $b=\alpha(x)\alpha(a^{-1})^\gamma x^{-1}$) for some $\gamma\in \Aut(G)_\alpha$, then $a,b$ are quasi-fused (rep. inverse-quasi-fused).
\end{definition}

Let $\ell_\alpha=G\rtimes \Aut(G)_\alpha$ be the semi-direct product. The element of $\ell_\alpha$ is denoted to be $(g,\gamma)$, where $\gamma\in \Aut(G)_\alpha$ and $g\in G$. The binary operation is defined as follows:
$$(g_1,\gamma_1)(g_2,\gamma_2):=(g_1\gamma_1(g_2),\gamma_1\gamma_2).$$
Let $a^{(g,\gamma)}:=\alpha(g)a^\gamma g^{-1}$. Then $\ell_\alpha$ indeed has an action on $G$ where
$$a^{(g_1,\gamma_1)(g_2,\gamma_2)}
= \alpha(g_1\gamma_1(g_2))a^{\gamma_1\gamma_2}\gamma_1(g_2^{-1})g_1^{-1}=a^{(g_1\gamma_1(g_2),\gamma_1\gamma_2)}.$$
We have a corollary as follows:

\begin{corollary}\label{abab}
Let $\alpha\in\Aut(G)$ be an involution. For $a,b\in G$,
then $\alpha(b^{-1})=\alpha(a^{-1})^{(g,\gamma)}$ if and only if $b=a^{(g,\gamma)}$, and
$b=\alpha(a^{-1})^{(g,\gamma)}$ if and only if $\alpha(b^{-1})=a^{(g,\gamma)}$.
That is, $\{a,\alpha(a^{-1})\}^{(g,\gamma)}=\{b,\alpha(b^{-1})\}$.
\end{corollary}

Then we have another revision of Definition \ref{qfiqf} as follows.
\begin{definition}
Let $\alpha\in\Aut(G)$ be an involution. For $a,b\in G$,
then $a,b$ are quasi-fused (rep. inverse-quasi-fused) if
$a^{(g,\gamma)}=b$ (rep. $\alpha(a^{-1})^{(g,\gamma)}=b$) for some $(g,\gamma)\in\ell_\alpha$.
\end{definition}

\begin{remark}
If $\alpha=1$ in Definition \ref{qfiqf}, then $b=ga^\gamma g^{-1}$ or $b=g(a^{-1})^\gamma g^{-1}$.
Thus the concept of quasi-fused and inverse-quasi-fused is equivalent to fused and inverse-fused, respectively.
It implies Definition \ref{qfiqf} is a generalization of fused and inverse-fused.
\end{remark}

Let $\wi{F}_\alpha(a)=\{\alpha(g)a^\gamma g^{-1}|g\in G,\gamma\in\Aut(G)_\alpha\}$ be the {\it quasi-fusion class}
with respect to $a$. If $\alpha$ is obvious, $\wi{F}_\alpha(a)$ is short for $\wi{F}(a)$, then $\omega_\alpha(G)=\wi{F}(1)$. The following properties will be displayed.

\begin{proposition}
Let $a,b\in G$ and $\alpha\in\Aut(G)$ be an involution.
\begin{enumerate}[(i).]
  \item If $\wi{F}_\alpha(a)\bigcap\wi{F}_\alpha(b)\neq\emptyset$, then they are equal.
  \item If $\alpha=\sigma_a$, then $\wi{F}_\alpha(a)=\{a\}$ and $\wi{F}_\alpha(x)=\wi{F}_\alpha(1)$, where $a\neq x\in G_\alpha$.
\end{enumerate}
\end{proposition}


\begin{proposition}\label{con}
$\wi{F}_\alpha(a)^\delta=\wi{F}_\beta(a^\delta)$ with $\alpha^\delta=\beta$.
\end{proposition}
\begin{proof}
If $x\in\wi{F}_\beta(a^\delta)$, then $x=\beta(g)(a^\delta)^\theta g^{-1}$, where $\theta\in\Aut(G)_\beta$.
It follows that $x=\alpha^\delta(g)a^{\theta\delta}g^{-1}=\delta\alpha\delta^{-1}(g)a^{\theta\delta}\delta\delta^{-1}(g^{-1})=
\delta(\alpha\delta^{-1}(g)a^{\delta^{-1}\theta\delta}\delta^{-1}(g^{-1}))$. Since $\theta\in\Aut(G)_\beta$, then
$\delta\alpha\delta^{-1}\theta=\theta\delta\alpha\delta^{-1}$, which implies $\alpha\delta^{-1}\theta\delta=\delta^{-1}\theta\delta\alpha$,
thus $\delta^{-1}\theta\delta\in\Aut(G)_\alpha$. Thus $\alpha\delta^{-1}(g)a^{\delta^{-1}\theta\delta}\delta^{-1}(g^{-1})\in\wi{F}_\alpha(a)$ and $x\in\wi{F}_\alpha(a)^\delta$, vice versa.
\end{proof}


According to Definition \ref{qfiqf}, we find it is still fuzzy for us to see the relation between two elements $a,b$ if they are quasi-fused or inverse-quasi-fused. It is indeed a trouble when we deal with the isomorphism problems of
generalized Cayley graphs. Fortunately, we find a clearer relationship
in some special cases.

\begin{theorem}\label{transform}
$\{b,\alpha(b^{-1})\}=\{a,\alpha(a^{-1})\}^{(g,\gamma)}$ if and only if $\sigma_g\gamma$ transforms $\{xa,a^{-1}x\}$ to $\{xb,b^{-1}x\}$, where $\alpha=\sigma_x$ and $(g,\gamma)\in\ell_\alpha$.
\end{theorem}
\begin{proof}
Let $\{b,\alpha(b^{-1})\}=\{a,\alpha(a^{-1})\}^{(g,\gamma)}$. According to Corollary \ref{abab}, then we have $a^{(g,\gamma)}=b$ or $\alpha(a^{-1})^{(g,\gamma)}=b$.

On the one hand, if $a^{(g,\gamma)}=b$ for some $(g,\gamma)\in\ell_\alpha$,
then $b=xgx^{-1}a^\gamma g^{-1}$. We have
\begin{eqnarray}
 && xb=gx^{-1}a^\gamma g^{-1} \\
 &\Leftrightarrow& xb=g(x^{-1}a)^\gamma g^{-1}\\
 &\Leftrightarrow& xb=(xa)^{\gamma'}
\end{eqnarray}
with $\gamma'=\sigma_g\gamma$, which implies that $xb$ and $xa$ are conjugate.
Similarly, we have $x\alpha(b^{-1})$ and $x\alpha(a^{-1})$, that is $b^{-1}x$ and $a^{-1}x$,
are also conjugate in $\Aut(G)$ with the same $\gamma'$.

On the other hand, if $b=\alpha(a^{-1})^{(g,\gamma)}$ for some $(g,\gamma)\in\ell_\alpha$,
then $b=xgx^{-1}\alpha(a^{-1})^\gamma g^{-1}$. We have
\begin{eqnarray}
 && xb=g(a^{-1})^\gamma x^{-1} g^{-1}  \\
 &\Leftrightarrow& xb=g(a^{-1}x)^\gamma g^{-1}\\
 &\Leftrightarrow& xb=(a^{-1}x)^{\gamma'}
\end{eqnarray}
with $\gamma'=\sigma_g\gamma$, which implies that $xb$ and $a^{-1}x$ are conjugate.
Similarly, we have $x\alpha(b^{-1})$ and $xa$, that is $b^{-1}x$ and $xa$, are also conjugate in $\Aut(G)$ with the same $\gamma'$.

Therefore, we can find that  $\Aut(G)$ transforms $\{xa,a^{-1}x\}$ to $\{xb,b^{-1}x\}$ no matter which case happens above.
The inverse is clear by (2)--(7) as above.
\end{proof}

\begin{corollary}
If $\alpha(a^{-1})=a, \alpha(b^{-1})=b, \alpha(c^{-1})=c$ and $\alpha(d^{-1})=d$, then $\{c,d\}=\{a,b\}^{(g,\gamma)}$ if and only if $\sigma_g\gamma$ transforms $\{xa,xb\}$ to $\{xc,xd\}$, where $\alpha=\sigma_x$ and $(g,\gamma)\in\ell_\alpha$.
\end{corollary}


Let $\Delta_\alpha:=\Inn(G)\Aut(G)_\alpha$. Then $\Delta_\alpha\leq\Aut(G)$.
Besides, $\Delta_\alpha$ and $\ell_\alpha$ are isomorphic and $\sigma_g\gamma\in \Delta_\alpha$, where $\gamma\in\Aut(G)_\alpha$.
Note that if $a=\alpha(a^{-1})$ for some $a\in G$, then $xa=a^{-1}x$. Two corollaries are introduced.


\begin{corollary}
Let $X_i=\GC(G,S_i,\alpha)$, where $S_i=\{a_i,\alpha(a_i^{-1})\},\,(i=1,2)$.
If $\alpha=\sigma_x$ for some $x\in G$ such that $xa_1$ and $xa_2$ are conjugate in $\Delta_\alpha$, then $X_1\cong X_2$.
\end{corollary}
\begin{proof}
Since $xa_1$ and $xa_2$ are conjugate in $\Delta_\alpha$, there is a $\sigma_g\gamma\in\Delta_\alpha$ such that $xa_2=(xa_1)^{\sigma_g\gamma}$.
It implies that $\{xa_1,a_1^{-1}x^{-1}\}^{\sigma_g\gamma}=\{xa_2,a_2^{-1}x^{-1}\}$. Thus $X_1\cong X_2$ by Theorem \ref{transform}.
\end{proof}

\begin{corollary}\label{}
$\ell_\alpha$ and $\Delta_\alpha$ are permutation isomorphic on $\bigcup_{i\in \pi_p}\Omega_x^*(G,i)$ and $\Pi_\alpha(G)$ respectively.
\end{corollary}
\begin{proof}
Let $\lambda:\kappa=\{a,\alpha(a^{-1})\}\mapsto\{xa,a^{-1}x\}$ and $\psi:\hbar=(g,\gamma)\mapsto \sigma_g\gamma$, then we will have
$\lambda(\kappa^\hbar)=\lambda(\kappa)^{\psi(\hbar)}$ by Theorem \ref{transform}.
\end{proof}

\begin{lemma}
Let $x\in G$ be an involution and $\alpha=\sigma_x$. Then
$\Aut(G)_{\omega_x^*(G)}=\Delta_\alpha$.
\end{lemma}
\begin{proof}
Suppose $\Psi=\omega_x^*(G)$. Since
$\Aut(G)_{\{\Psi\}}=\{\epsilon\in\Aut(G)|\Psi^\epsilon=\Psi\}=\{\epsilon\in\Aut(G)|x^\epsilon=x^h\,{\rm for\, any}\, h\in G\}$,
it implies $x^{\sigma_h\epsilon^{-1}}=x$, then $\sigma_h\epsilon^{-1}\in\Aut(G)_x$, thus $\epsilon\in\Delta_\alpha$. The inverse is obvious.
\end{proof}

\begin{theorem}
The symbols are defined as above, then we have the following results:
\begin{enumerate}[(1).]
  \item Let $\wi{F}=\{xa,a^{-1}x|a\in \wi{F}(1)\}$. Then $\wi{F}=\omega_x^*(G)$.
  \item $\Omega_x(G,i)$ and $\Omega_y(G,i)$ are conjugacy subsets if and
only if $x,y$ are conjugacy involutions.
  \item  $\Delta_\alpha$ has an action on $\Omega_x(G,i)$ for $i\big{|}|G|$. Especially, if $\Delta_\alpha$ is not transitive
  on $\Omega_x(G,2)$, then $\wi{F}$ is one of its orbits.
\end{enumerate}
\end{theorem}
\begin{proof}
(1). It can be checked that $\wi{F}=\omega_x^*(G)$ as for any $a\in \wi{F}(1)$, there is
$a=\alpha(g)g^{-1}=xgxg^{-1}$ for some $g\in G$.

(2). It can be obtained by the construction of $\Omega_x(G,i)$.

(3). It can be proved by Theorem \ref{transform} and (1).
\end{proof}







\section{The proof of Question \ref{ques2}}\label{answer1}
\begin{theorem}\label{fusiontwo}
Let $G$ be simple and finite. For any involutory automorphism $\alpha$ and $n\big{|}|G|$, $G$ has $P_e(n,2)$ in $\Aut(G)$ if $\Delta_\alpha$ acts transitive on $\Omega_\alpha^*(G,n)$.
\end{theorem}
\begin{proof}
Let $\alpha$ be an involutory automorphism, that is $\alpha=\sigma_x$ for some $x\in G$. Note that $\Omega_x^*(G,n)=\Omega_x(G,n)$ for $n\neq 2$ by (1) and $\Delta_\alpha$ is actually a setwise stabilizer of $\Aut(G)$. Therefore, if $\Delta_\alpha$ acts transitive on $\Omega_\alpha^*(G,n)$, then $\Aut(G)$ is the same as well for $n\neq 2$. If $n=2$, then $\Omega_x^*(G,2)=\Omega_x(G,2)-\omega_x^*(G)$ by (2). We need consider two cases.

Case 1: $C_G(2)\leq2$.

Note that $F_G(n)\leq C_G(n)$ for any $n\big{|}|G|$, then Case 1 is obvious.

Case 2: $C_G(2)\geq 3$.

We assume that $C_G(2)=3$ without loss of generality. Let $\{x,y,z\}$ be the elements of the involutory conjugacy classes, respectively. Thus $\Delta_\alpha$ (rep. $\Delta_\beta$, $\Delta_\gamma$ )is transitive on $\Omega_\alpha^*(G,2)$ (rep. $\Omega_\beta^*(G,2)$, $\Omega_\gamma^*(G,2)$), where $\alpha=\sigma_x,\,\beta=\sigma_y$ and $\gamma=\sigma_z$.
Therefore, there exists $a\in\Delta_\alpha$ (rep. $b\in\Delta_\beta$, $c\in\Delta_\gamma$) such that $y^a=z$ (rep. $x^b=z$, $x^c=y$).
It implies that
$\lg\Delta_\alpha,\Delta_\beta,\Delta_\gamma\rg$ acts transitive on $\Omega(G,2)$. Since $\lg\Delta_\alpha,\Delta_\beta,\Delta_\gamma\rg\leq\Aut(G)$, $\Aut(G)$ acts transitive on $\Omega(G,2)$ in this case. Therefore, $F_G(2)=1$ by \cite[Theorem 1.2]{Lich-CP-1}.
As a conclusion, $F_G(n)\leq 2$ for any $n\big{|}|G|$ in any case, that is $G$ has $P_e(n,2)$.
\end{proof}

\begin{theorem}\label{main1}
Let $G$ be simple and finite. For any involutory automorphism $\alpha$ and $n\big{|}|G|$, $\Delta_\alpha$ is transitive on $\Omega_\alpha^*(G,n)$  if and only if $G$ is $A_5$, $L_2(7)$, $A_6$, $L_2(8)$, $M_{11}$, $L_3(4)$, $Sz(8)$ or $M_{23}$.
\end{theorem}
\begin{proof}
If $\Delta_\alpha$ is transitive on any $\Omega_\alpha^*(G,n)$, then $F_G(n)\leq2$ for any $n\big{|}|G|$ by Theorem \ref{fusiontwo}. Hence
$G$ can only be the groups in \cite[Theorem 1.1]{Lich-CP-1}.

Let $G=A_n,\,(5\leq n\leq8)$. When $n=7,8$, we suppose $\alpha=\sigma_x$, where $x=(1\,2)(3\,4)$. Let $g=(1\,2\,3)$, $h=(1\,4)(2\,5\,3\,6)$. Hence, $g,f\notin \wi{F}(1)$. Further,
$xg=(1\,3\,4)$ and $xh=(1\,5\,3)(2\,4\,6)$, thus $xg$ and $xh$ (rep. $(xh)^{-1}$) are not fused as $\Aut(G)={\rm S}_n$.

Let $G$ be the group in \cite[Theorem 1.1]{Lich-CP-1}. Moreover, assume $G$ is not in \cite[Theorem 1.2]{Lich-CP-1} and not $A_n$ as above. Then we can always find two couples $\{a,a^{-1}\}$ and $\{b,b^{-1}\}$ with the same order such that $\{a,a^{-1}\}^\gamma\neq\{b,b^{-1}\}$ for any $\gamma\in\Aut(G)$ from the proof of \cite[Theorem 1.2]{Lich-CP-1}. Hence, $\Delta_\alpha$ is not transitive on the corresponding $\Omega_\alpha^*(G,n)$. Therefore, $G$ must be one of the eight groups above.

Conversely, if $G$ is any one of those eight groups, then $C_G(2)=1$ by \cite{CCNPW}, which implies that
$\Delta_\alpha=\Aut(G)$ and $\Omega_\alpha^*(G,2)=\emptyset$. Therefore, the sufficiency is obvious by \cite[Theorem 1.2]{Lich-CP-1}.
\end{proof}

\section{Several groups which are not restricted $m$-GCI-groups}\label{example}

\begin{lemma}\label{A6}
$A_6$ is not local $1$-GCI-group.
\end{lemma}
\begin{proof}
Let $G=A_6$ and $a=(1\,2)(3\,4)$, $b=(1\,2)$. Then we see that $a$ and $b$ induce a involutory inner automorphism and outer automorphism on $G$ by conjugation, respectively. Let $\alpha,\beta\in\Aut(G)$. And they are induced by $a$ and $b$. Note that $a\in G_{-\alpha}-\omega_\alpha(G)$ by Lemma \ref{+1}, so there exists $\GC(G,S,\alpha)$ where $S=\{a\}$. By computation, we obtain $|G_\beta|=24$ as the elements in $G_\beta$ are even permutations in $S_\Omega$ where $\Omega=\{3,4,5,6\}$ or like `$(12)\cdot k$' where $k$ is an odd permutations in $S_\Omega$, thus $|\omega_\beta(G)|=\frac{|G|}{|G_\beta|}=15$.
For $x\in G_{-\beta}$, it should be the element satisfying $(12)x(12)x=1$. Hence, if the element is like $(x_1\,x_2)(x_3\,x_4)$, $(1\,2)(x_1\,x_2)$ or $(1\,2)(x_1\,x_2\,x_3\,x_4)$, where $x_i\in\{3,4,5,6\}$, then it is contained in $G_{-\beta}$. It means $|G_{-\beta}|\geq 15$. Note that $(1\,2)(1\,2\,x_5)(1\,2)(1\,2\,x_5)=1$ where $x_5\in \{3,4,5,6\}$, it implies that $|G_{-\beta}|>15$ and $|G_{-\beta}|>|\omega_\beta(G)|$. Thus, there exists one generalized Cayley graph like $\GC(G,\{k\},\beta)$, where $k\in G_{-\beta}-\omega_\beta(G)$. Hence, $G$ is not local 1-GCI-group.
\end{proof}

\begin{lemma}\label{L27}
$L_2(7)$ is not local 2-GCI-group.
\end{lemma}
\begin{proof}
Let $G=L_2(7)$. Then $G\cong L_3(2)=\lg x_1,x_2|x_1^2=x_2^3=(x_1x_2)^7=[x_1,x_2]^4=1\rg$ by \cite{CCNPW}.
Let $x_1=(1\,4)(6\,7),x_2=(1\,3\,2)(4\,7\,5)$.
By \cite{ATLAS-Rep} and Magma \cite{BCP}, we can see that $\alpha:=\sigma_{x_1}$ and $\beta:\left\{\begin{array}{l}
                                                                x_1\mapsto x_1, \\
                                                                x_2\mapsto x_2^{-1},
                                                              \end{array}\right.
$ are inner and outer automorphism of $G$ respectively. Let $S_1=\{(1\,5\,2\,6\,3\,7\,4), (1\,6\,3\,7\,2\,5\,4)\}$ and
$S_2=\{(1\,4\,5\,3)(2\,7),(1\,3\,5\,4)(2\,7)\}$. By Magma \cite{BCP}, $\GC(G,S_1,\alpha)\cong\GC(G,S_2,\beta)$. However, these two involutory automorphisms are not conjugate, it follows that $L_2(7)$ is not local 2-GCI-group.
\end{proof}

\begin{lemma}\label{L34}
$L_3(4)$ is not local 2-GCI-group.
\end{lemma}
\begin{proof}
Let $G=L_3(4)$. Then $G=\lg x_1,x_2|x_1^2=x_2^4=(x_1x_2)^7=(x_1x_2^2)^5=(x_1x_2x_1x_2^2)^7=(x_1x_2x_1x_2x_1x_2^2x_1x_2^{-1})^5=1\rg$ by \cite{CCNPW}.
Let $x_1=(1\,2)(4\,6)(5\,7)(8\,12)(9\,14)(10\,15)(11\,17)(13\,19)$ and
$x_2=(2\,3\,5\,4)(6\,8\,13\,9)(7\,10\,16\,11)(12\,18)(14\,20\,21\,15)(17\,19)$.
By \cite{ATLAS-Rep} and Magma \cite{BCP}, there is an inner involutory automorphisms $\alpha:=\sigma_{x_1}$ and an outer involutory automorphisms $$\beta:\left\{\begin{array}{l}
                      x_1\mapsto (1\,15)(2\,10)(4\,17)(5\,13)(6\,11)(7\,19)(8\,9)(12\,14), \\
                      x_2\mapsto (2\,13\,16\,21)(3\,9\,11\,15)(4\,8\,10\,20)(5\,6\,7\,14)(12\,17)(18\,19).
                    \end{array}\right.
$$
Set $$S_1:=\left\{\begin{array}{c}
          (1\, 7\, 17\, 14\, 3)(2\, 21\, 9\, 19\, 4)(5\, 10\, 16\, 8\, 11)(6\, 13\, 12\, 18\, 15), \\
          (1\, 6\, 13\, 14\, 21)(2\, 3\, 9\, 11\, 5)(4\, 10\, 18\, 8\, 19)(7\, 17\, 12\, 16\, 15)
        \end{array}\right\};$$
$$S_2:=\left\{\begin{array}{c}(1\, 19)(2\, 8)(3\, 16)(4\, 15)(7\, 14)(11\, 13)(12\, 17)(20\, 21),\\
(1\, 12)(2\, 11)(3\, 20)(4\, 14)(7\, 15)(8\, 13)(16\, 21)(17\, 19) \end{array}\right\}.$$
It follows that $\GC(G,S_1,\alpha)\cong\frac{|G|}{4}C_4\cong\GC(G,S_2,\beta)$. However, there exists no $g\in G$ and $\gamma\in\Aut(G)$ satisfying Definition \ref{D-G}. Thus $L_3(4)$ is not a local 2-GCI-group.
\end{proof}

\section{Local $m$-GCI-groups}\label{answer2}
Now we will classify all local $m$-GCI simple groups for $m=1$ or $2$. Let $X_i=\Cay(G,S_i)$, where $S_i=\{x_i,x_i^{-1}\},\,(i=1,2)$. It is well known that if $o(x_1)=o(x_2)$, then $X_1\cong X_2$, vice versa. However, the situation of generalized Cayley graphs is complicated, so we will
characterize 2-valent generalized Cayley graphs next.

\begin{proposition}\label{odd}
Let $G$ be simple and finite. Suppose $X=\GC(G,S,\alpha)$ where $\alpha=\sigma_x$ for some $x\in G$ and $S=\{a,\alpha(a^{-1})\}$.
Let $o(xa)=n$. Then $X\cong \frac{|G|}{2n}C_{2n}$ for $n$ is odd.
\end{proposition}
\begin{proof}
Assume $n$ is odd. For $g\in G$, the path on $g$ is like
$$\cdots-g-xg(xa)-g(xa)^2-xg(xa)^3-\cdots-g(xa)^{n-1}-\cdots,$$
then $xg(xa)^{2k-1}\neq g$ as $g^{-1}xg$ is an involution and $g(xa)^{2k}\neq g$ as $n\neq 2k$, so the path is not going to be closed before $g(xa)^{n-1}$, then we will have the left path is as follows:
$$\cdots-g(xa)^{n-1}-xg-gxa-xg(xa)^2-\cdots-xg(xa)^{n-1}-\cdots.$$
Note that $xg(xa)^{2k}\neq g$ as $n$ is odd for $2k<n$. Furthermore, when $n>2k-1$, $g(xa)^{2k-1}\neq g$. Thus $g(xa)^{2k-1}=g$ only if $n=2k-1$. Therefore, the length of the cycle is $2n$, which looks like as follows:
$$g-xg(xa)-g(xa)^2-\cdots-g(xa)^{n-1}-xg-gxa-xg(xa)^2-\cdots-xg(xa)^{n-1}-g.$$
\end{proof}

\begin{corollary}\label{aa5}
Let $G=A_5$ and $x$ be an involution of $G$. If $\alpha=\sigma_x$ for some $x\in G$, then $\GC(G,S,\alpha)$ is
vertex-transitive for $S=\{a,b\}$.
\end{corollary}
\begin{proof}
Suppose that $x=(1\,2)(3\,4)$. Therefore, $|\omega_\alpha(G)|=\frac{|G|}{|G_\alpha|}=15$ and
$G_{-\alpha}-\omega_\alpha(G)=\{x\}$. It implies that if $S=\{a,b\}$, then we must have $b=\alpha(a^{-1})\neq a$.
Since $S\in\Pi_\alpha(G)$, $\Pi_\alpha(G)\bigcap\omega_\alpha(G)=\emptyset$, it implies $xa$ can not be the conjugacy elements of $x$,
thus $o(xa)=3$ or $5$. By Proposition \ref{odd}, it follows $\GC(G,S,\alpha)\cong \frac{|G|}{2n}C_{2n}$, where $n=o(xa)$.
\end{proof}

\begin{remark}
Note that $\Cay(G,S)\cong\frac{|G|}{|\lg S\rg|}\Cay(\lg S\rg,S)$. Then all the 2-valent Cayley graphs on $A_5$ are $20C_3$ and $12C_5$. According to Corollary \ref{aa5}, we can obtain some other vertex-transitive graphs by generalized Cayley graphs except for the vertex-transitive graphs constructed by Cayley graphs on the group.
\end{remark}

\begin{proposition}\label{ncong}
Let $X_i=\GC(G,S_i,\alpha)$ and $S_i=\{a_i,\alpha(a_i^{-1})\}$ where $i=1,2$. Let $\alpha=\sigma_x$ for some $x\in G$. If $o(xa)\neq o(xb)$, then $X_1\ncong X_2$.
\end{proposition}
\begin{proof}
Let $o(xa_1)=m_1$ and $o(xa_2)=m_2$.
We assume that $m_1$ and $m_2$ are different odds, then $X_1\ncong X_2$ by Proposition \ref{odd}.
If $m_i$ $(i=1,2)$ is even, by the proof of Proposition \ref{odd}, then we can see that the length of all the cycles in $X_i$ is $m_i$ when $g(xa)^{\frac{m_i}{2}}\neq g$ or $\frac{m_i}{2}$ when $g(xa)^{\frac{m_i}{2}}=g$. It follows that if at least one of
$m_1$ and $m_2$ are evens, then $X_1\ncong X_2$.
\end{proof}

\begin{lemma}\label{cycle}
Let $\alpha\in\Aut(G)$ be an involution for a simple and finite group $G$. Suppose $X=\GC(G,S,\alpha)$ and $S=\{a,b\}$ in which $\alpha(a^{-1})=a$ and $\alpha(b^{-1})=b$. Then
the cycle on any vertex $g$ is like
  $$g-\alpha(g)a-ga^{-1}b-\cdots-\alpha(g)b-g $$ and $\alpha(g)ab^{-1}a\cdots ab^{-1}a=\alpha(g)b$.
  Moreover, $\GC(G,S,\alpha)\cong\Cay(G,\{a^{-1}b,b^{-1}a\})$.

\end{lemma}
\begin{proof}
Note that $\{g,\alpha(g)s\}\in E(X)$ for
any $g\in G$, where $s=a$ or $b$.
If $a=\alpha(a^{-1})$ and $b=\alpha(b^{-1})$, then the other vertex adjecent to $\alpha(g)a$ is $g\alpha(a)b$ as $g\alpha(a)a=g$.
Similarly, the other vertex adjecent to $g\alpha(a)b$ is $\alpha(g)a\alpha(b)a$.
So the cycle is like
$$g-\alpha(g)a-ga^{-1}b-\alpha(g)ab^{-1}a-ga^{-1}ba^{-1}b-\cdots-\alpha(g)b-g.$$
Then $ga^{-1}b\cdots a^{-1}b=\alpha(g)b$ or $\alpha(g)ab^{-1}a\cdots ab^{-1}a=\alpha(g)b$. If
$ga^{-1}b\cdots a^{-1}b=\alpha(g)b$, note that the next vertex connecting this vertex is $g$, then
$\alpha(ga^{-1}b\cdots a^{-1}b)a=g$, this implies that $g\alpha(b)a=g$, which is a contradiction.
\end{proof}

Now we focus on the local $m$-GCI-groups for $m=1$ or $2$. We display the description of $1$-GCI-groups.

\begin{theorem}
Let $G$ be simple and finite. If it is a local 1-GCI-group, then $G$ satisfies conditions:
\begin{enumerate}
  \item $F_G(2)=1$;
  \item for any outer involutory automorphism $\alpha$, $G_{-\alpha}=\omega_\alpha(G)$.
\end{enumerate}
\end{theorem}
\begin{proof}
Suppose $G$ is a local 1-GCI-group. Let $\GC(G,S_1,\alpha)\cong\GC(G,S_2,\beta)$, where $|S_i|=1\,(i=1,2)$ and $\alpha$ be an involutory inner automorphism. Then $\beta$ is either an involutory inner automorphisms or an outer automorphism.
If $\beta$ is an inner automorphism, then the corresponding involutions are conjugate in $\Aut(G)$. If $\beta$ is an outer automorphism, then there is no $\GC(G,S_2,\beta)$, where $|S_2|=1$, otherwise it will be a contradiction. It implies $G_{-\beta}=\omega_\beta(G)$.
\end{proof}

\begin{theorem}\label{results-1}
There is no $m$-GCI-groups in finite nonabelian simple groups.
\end{theorem}
\begin{proof}
By Definition \ref{D-G}, for all $k\leq m$, a $m$-GCI-group is a $k$-GCI-group. And $C(G,2)\geq2$ for any group $G$ which is simple and finite.  Therefore, there exist $\GC(G,\{x\},\sigma_g)$ and $\GC(G,\{x\},1)$  for any involution $x\in G$ according to Lemma \ref{+1}. Since $\GC(G,\{x\},1)\cong\frac{|G|}{2}K_2\cong\GC(G,\{x\},\sigma_x)$ and $G$ is not generalized Cayley isomorphism, we have $G$ is not 1-GCI-group.
\end{proof}

\begin{lemma}\label{A5}
$A_5$ is local $2$-GCI-group.
\end{lemma}
\begin{proof}
Let $x=(1\,2)(3\,4)$ and $y=(1\,2)$. Then $x$ (rep. $y$) reduces a involutory automorphism $\alpha$ (rep. $\beta$) of $A_5$, where $\alpha=\sigma_x$. Let $G=A_5$. Then $x\in G$ and there exists the unique $\GC(G,\{x\},\alpha)$ of valency 1 by Lemma \ref{+1} respecting to $\alpha$. Furthermore, it is the unique generalized Cayley graph of valency 1 of $G$.
To graphs of valency 2, if $X_1\cong X_2$, then $o(xa)=o(xb)$ by Proposition \ref{ncong}, so there is $\delta\in\Aut(G)$ such that Theorem \ref{transform} is valid as $F_G(n)\leq 2$ for any $n\big{|}|G|$ by \cite[Theorem 1.3]{Lich-CP-1}.

For the graphs reduced by $\beta$ of valency 2, we can find that there are 15 generalized Cayley graphs isomorphic to $15C_4$ and 10 generalized Cayley graphs isomorphic to $2C_3\cup9C_6$ by Magma \cite{BCP}.
Let $\GC(G,S_1,\beta)\cong\GC(G,S_2,\beta)$.
If $\beta(g)S_1^\delta g^{-1}=S_2$, it implies $g((1\,2)S_1)^\delta g^{-1}=(1\,2)S_2$, that is $((1\,2)S_1)^\sigma=(1\,2)S_2$, where $\sigma\in\Aut(G)$, which means the question transforms to whether two subsets $(1\,2)S_1$ and $(1\,2)S_2$ are equivalent. By Magma \cite{BCP}, we find that there are two orbits of such subsets, one is of length 15 and another is 10, which coincides with fifteen $15C_4$s and ten $2C_3\cup9C_6$s respectively.

By Corollary \ref{aa5}, we can see that $\GC(G,\{a,\alpha(a^{-1})\},\alpha)\cong \frac{|G|}{2n}C_{2n}$, where $o(xa)=n$, for any $a$ respect to $A_5$.
It follows that any two graphs about $\alpha$ and $\beta$ respectively are not isomorphic to each other as $n$ can not be $2$ since those elements are in $\omega_\alpha(G)$. Therefore, $A_5$ is local.
\end{proof}

\begin{lemma}\label{gci-necessary}
Suppose $G$ is simple and finite, and $C(G,2)\geq1$. Let $G$ be a local 2-GCI-group. Then for any $n\big{|}|G|$, $\Delta_\alpha$ is transitive on $\Omega_\alpha^*(G,n)$.
\end{lemma}
\begin{proof}
Let $G$ be local 2-GCI-group. Let $X_i=\GC(G,S_i,\alpha)$, where $S_i=\{a_i,\alpha(a_i^{-1})\},\,(i=1,2)$. If $X_1\cong X_2$, then  we have $g\in G$ and $\gamma\in\Aut(G)_\alpha$ such that $S_2=\alpha^\gamma(g)S_1^\gamma g^{-1}$. It implies $S_2=S_1^{(g,\gamma)}$. By Theorem \ref{transform}, we have $\{xa,a^{-1}x\}^{\sigma_g\gamma}=\{xb,b^{-1}x\}$ and $o(xa)=o(xb)$. Therefore, for all $n$, $\Delta_\alpha$ is transitive on $\Omega_\alpha^*(G,n)$ .
\end{proof}



\begin{theorem}\label{results-2}
Suppose $G$ is simple. Then $G$ is local 2-GCI-group if and only if $G$ is one of groups as follows:
$$A_5, L_2(8), M_{11}, Sz(8), M_{23}.$$
\end{theorem}
\begin{proof}
Assume $G$ is local 2-GCI-group. Then $G$ are those groups in Theorem \ref{main1} by Theorems \ref{gci-necessary} and \ref{main1}. By Lemmas \ref{A6}, \ref{L27} and \ref{L34}, $G$ can not be $A_6$, $L_2(7)$ or $L_3(4)$, then $G$ can only be the left five groups. Conversely, by Lemma \ref{A5}, if $G$ is $A_5$, it is a local 2-GCI-group. If $G$ is one of $L_2(8)$, $Sz(8)$, $M_{11}$ and $M_{23}$, we can see that $G$ has no outer involutory automorphisms and only one involutory conjugacy class from Table 1. It imples that $\Omega_x(G,2)=\omega_x^*(G)$ and $\Omega_x^*(G,2)=\emptyset$ for $x$ an involution of $G$. Moreover, we only need to consider generalized Cayley graphs induced by $\alpha=\sigma_x$. By
Proposition \ref{ncong}, if $X_1\cong X_2$, then $o(xa)=o(xb)$.
Hence, there exist $\{xa,a^{-1}x^{-1}\},\, \{xb,b^{-1}x^{-1}\}\in \Omega(G,i)$ for $i\neq 2$. It implies $\{xa,a^{-1}x^{-1}\}^\delta=\{xb,b^{-1}x^{-1}\}$ for some $\delta\in\Aut(G)$ by Theorem \ref{main1}.
By theorem \ref{transform}, we can see that there exist $(g,\gamma)\in\ell_\alpha$ mapping
$\{a,\alpha(a^{-1})\}$ to $\{b,\alpha(b^{-1})\}$. Therefore, $G$ is local 2-GCI-group.
\begin{table}[htbp]\label{3}
\centering

\begin{tabular}{c|cccccccc}
\hline
$G$ & $L_2(8)$ & $M_{11}$ & $Sz(8)$ & $M_{23}$ \\
\hline\hline
$Out(G)$ & 3 & 1  & 3 & 1  \\
$C_G(2)$ & 1 & 1 & 1 & 1 \\
\hline
\end{tabular}
\caption{Information of four groups}
\end{table}
\end{proof}

\textbf{Proof of Theorem \ref{main-re}:} Combining Theorems \ref{results-1} and \ref{results-2}, we will get Theorem \ref{main-re} immediately.

\section{Further work}
In the future, we will consider the classification of local $m$-GCI-groups on simple groups when $m=1$ and $m\geq 3$.
Furthermore, we would like to consider the analogous problem on some finite solvable groups.

\end{document}